\documentclass[12pt,reqno]{amsart}
\usepackage{amssymb,mathrsfs,graphicx}
\usepackage{mathtools}
\usepackage{pifont,color}
\usepackage[margin=3cm]{geometry}
\title[Nonlocal traffic flow models]{
  Sharp critical thresholds for a class of nonlocal traffic flow models}

\author[Thomas Hamori]{Thomas Hamori}
\address[Thomas Hamori]{\newline Department of Mathematics, \ 
 University of South Carolina, Columbia, SC 29208, USA}
\email{thamori@email.sc.edu}

\author[Changhui Tan]{Changhui Tan}
\address[Changhui Tan]{\newline Department of Mathematics, \ 
 University of South Carolina, Columbia, SC 29208, USA}
\email{tan@math.sc.edu}



\thanks{\emph{Acknowledgment.} This work is supported by the NSF grants
DMS \#1853001 and \#2108264, and a UofSC ASPIRE I grant (2020). }

\numberwithin{equation}{section}

\newtheorem{theorem}{Theorem}[section]
\newtheorem{lemma}{Lemma}[section]

\newtheorem{proposition}{Proposition}[section]
\newtheorem{remark}{Remark}[section]

\def\pa{\partial}
\def\rhob{\widetilde{\rho}}
\def\R{\mathbb{R}}
\def\rhoM{\rho_M}
\def\KM{\overline{K}}
\def\rhoc{\rho_c}
\def\d{\mathrm{d}}

\begin{document}
\allowdisplaybreaks

\begin{abstract}
We study a class of traffic flow models with nonlocal look-ahead
interactions. The global regularity of solutions depend on the
initial data. We obtain sharp critical threshold conditions that
distinguish the initial data into a trichotomy: subcritical initial
conditions lead to global smooth solutions, while two types of
supercritical initial conditions lead to two kinds of finite time
shock formations.
The existence of non-trivial subcritical initial data 
indicates that the nonlocal look-ahead interactions can help avoid
shock formations, and hence prevent the creation of traffic jams.
\end{abstract}

\maketitle 

\section{Introduction}
The history of mathematical theory of traffic flow dates back to the
1920s. Many successful models have been proposed and studied to
understand the interactions and the emergent behaviors of vehicles on
the road.

One popular class of macroscopic traffic flow models are based on the
continuum description of the dynamics of the traffic density 
\begin{equation}\label{eq:continuity}
  \pa_t\rho+\pa_x(f(\rho))=0,\quad f(\rho)=\rho u(\rho).
\end{equation}
Here, $f$ is known as the \emph{flux}, which depends on local traffic
density $\rho=\rho(t,x)$. The traffic velocity $u$ is modeled through the
relation $u=u(\rho)$. A fundamental assumption is that $u$ is a
decreasing function in $\rho$, meaning vehicles slow down as 
traffic density increases.

A celebrated model under this framework is the
Lighthill-Whitham-Richards (LWR) model
\cite{lighthill1955kinematic,richard1956highway}, where
the velocity $u(\rho)=1-\rho$ decays linearly in $\rho$. The corresponding
flux reads
\begin{equation}\label{eq:LWR}
  f(\rho)=\rho(1-\rho).
\end{equation}
The LWR model successfully captures the phenomenon of shock formation,
which is responsible for the creation of traffic jams.

The flux in \eqref{eq:LWR} is concave and symmetric (with respect to
$\rho=\frac12$). However, statistical data from real-world traffic
networks suggest that the flux should be neither concave nor
symmetric. Rather, observed empirical fluxes are right-skewed and
become convex when the density is large, see
e.g. \cite{drake1965statistical}. In particular, a family of
fluxes were introduced in \cite{pipes1966car} that better fit the data
\begin{equation}\label{eq:fluxfJ}
  f_J(\rho)=\rho(1-\rho)^J,\quad J>0.
\end{equation}
For $J>1$, the flux $f_J$ is right-skewed, and switches from concave
to convex at a point $\rhoc=\frac{2}{J+1}\in(0,1)$.

In this paper, we consider a general class of fluxes with the
following hypotheses
\begin{equation}\label{eq:f}
 f\in C([0,1])\cap C^\infty([0,1)),\,\, f(0)=f(1)=0,\,\, f'(0)>0,\,\, f''(\rho)\begin{cases}
    <0& \rho\in[0,\rhoc),\\
    >0& \rho\in(\rhoc,1),
 \end{cases}
\end{equation}
with a parameter $\rhoc\in(0,1]$.
The assumptions in \eqref{eq:f} cover two scenarios of our concern.
First, when $\rhoc=1$, the flux is concave in $[0,1]$.
Examples include the flux \eqref{eq:LWR} in the LWR model, as well as 
fluxes in \eqref{eq:fluxfJ} with $J\in(0,1]$.
Second, when $\rhoc\in(0,1)$, the convexity of $f$ changes at $\rhoc$.
The fluxes in \eqref{eq:fluxfJ} with $J>1$ lie in this category.

The system \eqref{eq:continuity} with flux \eqref{eq:f} is a scalar
conservation law. The behaviors of global solutions have been
well-studied, see e.g. the book \cite{dafermos2016hyperbolic}. 
In particular, the system develops shock singularity in finite time, for any
generic smooth initial data that is not monotone decreasing.

We are interested in the following class of traffic flow models
with nonlocal look-ahead interaction
\begin{equation}\label{eq:main}
  \pa_t\rho+\pa_x\big(f(\rho)e^{-\rhob}\big)=0,\quad \rhob(t,x) =
  \int_0^\infty K(y)\rho(t,x+y)\,dy.
\end{equation}
Here, the term $e^{-\rhob}$ is known as the Arrhenius-type
\emph{slowdown factor}.
$\rhob$ represents the heaviness of the traffic ahead, weighted by a
kernel $K$. 

The system \eqref{eq:main} was first introduced by 
Sopasakis and Katsoulakis \cite{sopasakis2006stochastic}
where the flux $f$ is taken as in the LWR model \eqref{eq:LWR}, and
the interaction kernel
\begin{equation}\label{eq:K-SK}
  K(x)=1_{[0,L]}(x),
\end{equation}
where $1_E$ denotes the indicator function of the set $E$.
They formally derived \eqref{eq:main} from a microscopic cellular
automata (CA) model.
In the SK model, the look-ahead distance is $L$ and the weight is a
constant.
Another class of kernels has been studied numerically in
\cite{kurganov2009non} where
\begin{equation}\label{eq:K-linear}
  K(x)=\begin{cases}1-\frac{x}{L}&0<x<L,\\ 0&x\ge L.\end{cases}
\end{equation}
Finite time shock formations were observed in both models. The
so-called wave breaking phenomenon was studied in \cite{lee2015thresholds}.

Lee in \cite{lee2019thresholds} proposed and studied \eqref{eq:main}
where the flux is taken as \eqref{eq:fluxfJ} with $J=2$. The
non-concave-convex flux can lead to different types of shock formations.
Later in \cite{sun2020class}, the system was derived from a class of
CA models. An intriguing observation was that the parameter $J$ in
\eqref{eq:fluxfJ} corresponds to the number of cells a car moved in
one step of the microscopic dynamics.

The global wellposedness of \eqref{eq:main} and related nonlocal
traffic flow models have been extensively studied under the framework
of nonlocal conservation laws.
The theory of entropic weak solutions has been
established in
\cite{bressan2020traffic,colombo2018nonlocal,keimer2017existence,keimer2018nonlocal}.
These solutions can be discontinuous, allowing the formation of
shocks.

One challenging question is \emph{whether \eqref{eq:main} admits
  global smooth solutions}.
In other words, the question asks whether the nonlocal slowdown
interaction can help prevent shock formations, and consequently avoid
the creation of traffic jams.

A postive answer was given in \cite{lee2019sharp} in a special case
when the flux $f$ is \eqref{eq:LWR}, and the interaction kernel $K$ is
\eqref{eq:K-SK} with  look-ahead distance $L=\infty$, namely
\begin{equation}\label{eq:ourK}
  K(x) = 1_{[0,\infty)}(x),
\end{equation}
and correspondingly
\begin{equation}\label{eq:rhobar}
\rhob(t,x)= \int_x^\infty\rho(t,y)dy.
\end{equation}
A sharp \emph{critical threshold} on the initial data was established
that distinguishes the global behavior of the solutions:
subcritical initial data lead to global smooth solutions while
supercritical initial data lead to finite-time shock formations.
Such critical threshold phenomenon has
been studied in the context of Eulerian dynamics, including the
Euler-Poisson equations
\cite{engelberg2001critical,lee2013thresholds,tadmor2003critical},
the Euler-alignment equations
\cite{carrillo2016critical,tadmor2014critical,tan2020euler}, and more
systems of conservation laws \cite{bhatnagar2021critical,lee2019sharp,tadmor2021critical}.

In this paper, we study the critical threshold phenomenon for
\eqref{eq:main} with the general class of fluxes in \eqref{eq:f}.
Our first result is a generalization of \cite{lee2019sharp},
considering concave fluxes.
\begin{theorem}\label{thm:concave}
  Consider equation \eqref{eq:main} with smooth initial data
  $\rho_0\in L_+^1\cap H^k(\R)$ with $k>3/2$ and $\rho_0(x)\le\rhoM<1$.
  Suppose the flux $f$ is concave, satisfying \eqref{eq:f} with
  $\rhoc=1$.
  Suppose the nonlocal term $\rhob$ satisfies \eqref{eq:rhobar}.
  Then there exists a function $\sigma: [0,1]\to [0,\infty)$ such that 
  \begin{itemize}
  \item If the initial data is subcritical, satisfying 
    \[
      \rho_0'(x)\le\sigma(\rho_0(x)), \quad \forall~ x\in\R,
    \]
    then there exists a global smooth solution, namely for any $T>0$,
    \begin{equation}\label{eq:rhoreg}
      \rho\in C\big([0,T]; L^1_+\cap H^k(\R)\big).
    \end{equation}
  \item If the initial data is supercritical, satisfying
    \[\exists~ x_0\in\mathbb{R} \quad\text{s.t.}\quad \rho_0'(x_0)>\sigma(\rho_0(x_0)),\]
    then the solution must blow up in finite time. More precisely,
    there exists a location $x\in\R$ and a finite time $T_*>0$ such that
   \[\lim_{t\to T_*-}\pa_x\rho(t,x)=+\infty.\]
\end{itemize}
\end{theorem}

\begin{remark}
 Theorem \ref{thm:concave} recovers the result in \cite{lee2019sharp}
 when taking the flux $f$ in \eqref{eq:LWR}. 
 The left graph in Figure \ref{fig:CT} illustrates the shape of the
 threshold function $\sigma$. It can be constructed via the procedure
 described in Theorem \ref{thm:sigma}.
 
 Note that the subcritical region allows $\rho_0'(x)$ to take positive values.
 Hence, there is a family of non-monotone decreasing initial data that
 do not lead to shock formations. This provides a strong indication
 that the nonlocal look-ahead interaction can help preventing the creation
 of traffic jams, for subcritical initial configurations.
\end{remark}

\begin{figure}[ht]
  \includegraphics[width=.45\linewidth]{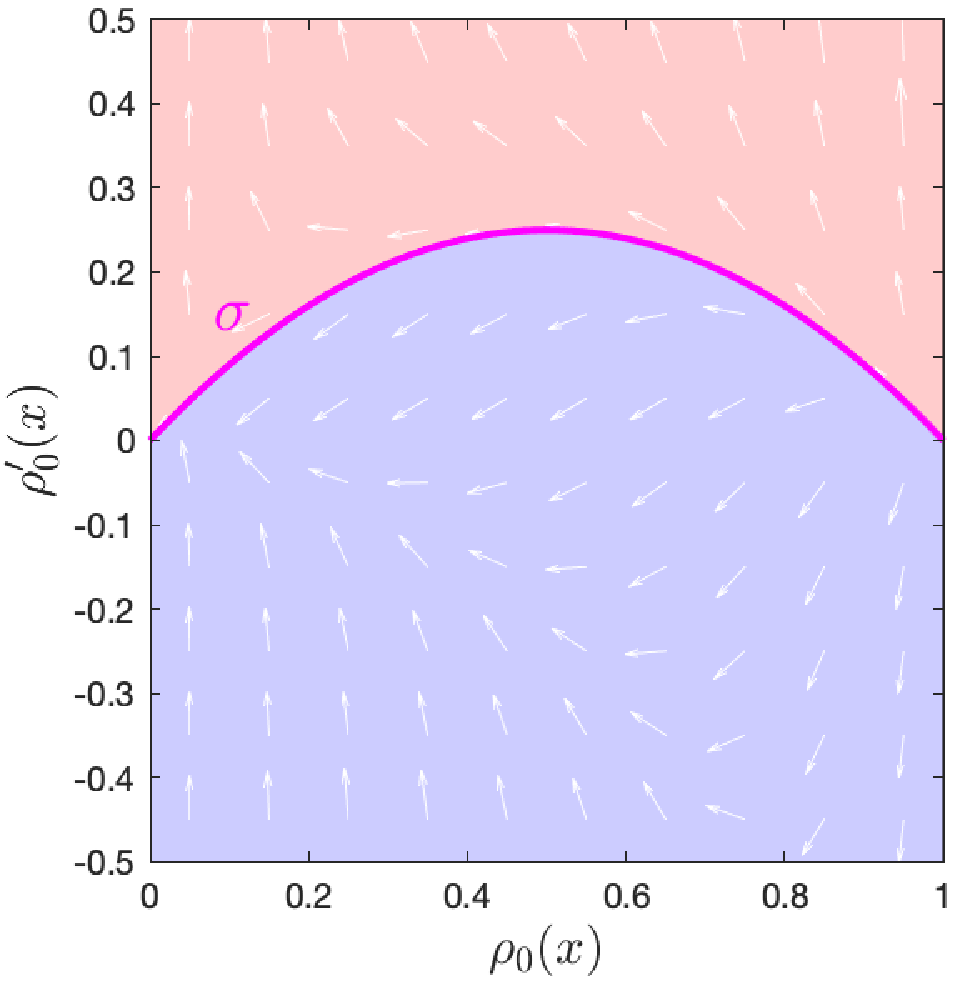}\quad
  \includegraphics[width=.45\linewidth]{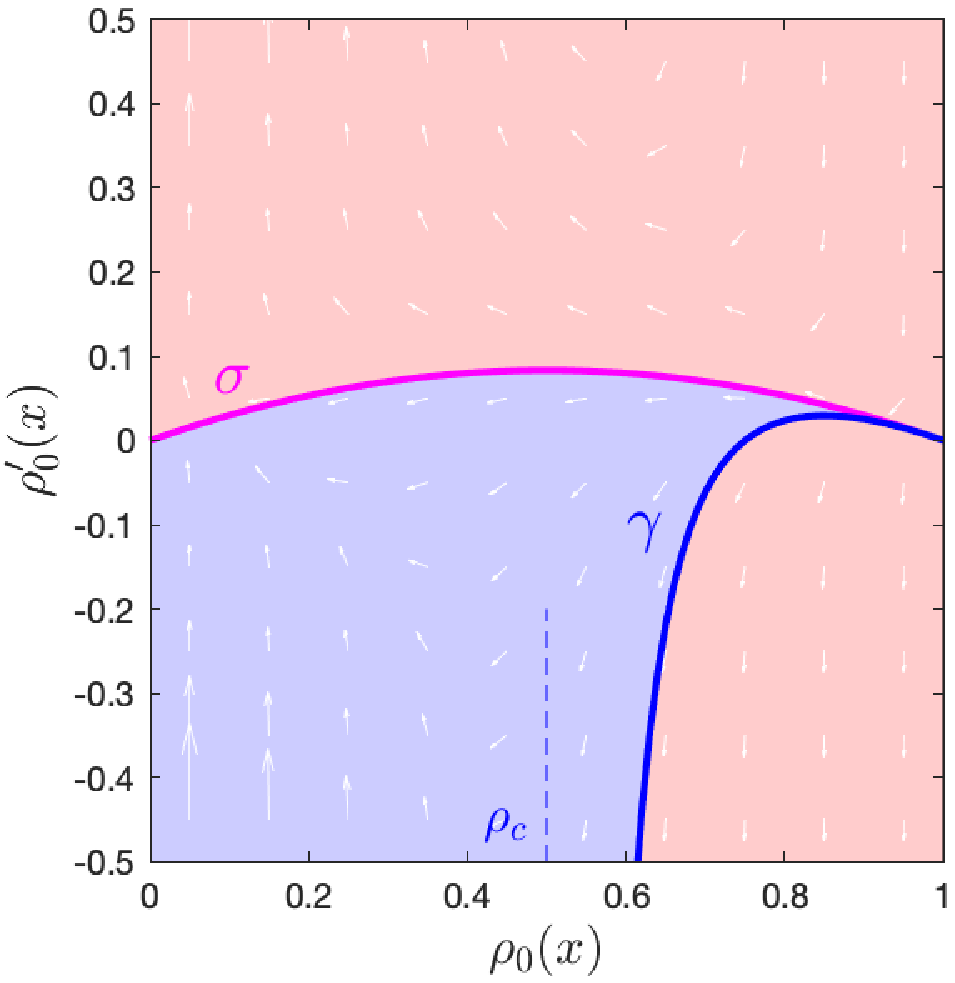}
  \caption{Illustration of the critical thresholds.
  Left: when $f$ is concave, the region above $\sigma$ is supercritical,
  and the region below $\sigma$ is subcritical.
  Right: when $f$ switches from concave to convex at $\rhoc$, the region
  above $\sigma$ is type I supercritical, the region below $\gamma$ is
  type II supercritical, and the remaining region is subcritical.}\label{fig:CT}
\end{figure}

The next main result concerns fluxes that are not concave. The lack of
concavity leads to a major difference in the global behaviors of the
solutions. In particular, there are two different types of shock
formations. There is a trichotomy on initial data that lead to global
regularity and two types of finite time blowup. The following theorem
provides a sharp characterization on the threshold conditions.

\begin{theorem}\label{thm:main}
Consider equation \eqref{eq:main} with smooth initial data
  $\rho_0\in L_+^1\cap H^k(\R)$ with $k>3/2$ and $\rho_0(x)\le\rhoM<1$.
  Suppose the flux $f$ satisfies \eqref{eq:f} with
  $\rhoc<1$, that is, $f$ is concave on $[0,\rhoc]$ and convex on $[\rhoc,1]$.
  Suppose the nonlocal term $\rhob$ satisfies \eqref{eq:rhobar}.
  Then there exists two threshold functions $\sigma$ and $\gamma$ such that 
\begin{itemize}
\item If the initial data is subcritical, satisfying 
  \[
    \gamma(\rho_0(x))<\rho_0'(x)\le\sigma(\rho_0(x)), \quad \forall~
    x\in\R,
  \]
  then there exists a global smooth solution $\rho$ satisfying \eqref{eq:rhoreg}.
  \item If the initial data is type I supercritical, satisfying
    \[\exists~ x_0\in\mathbb{R} \quad\text{s.t.}\quad \rho_0'(x_0)>\sigma(\rho_0(x_0)),\]
    then the solution must blow up in finite time. More precisely,
    there exists a location $x\in\R$ and a finite time $T_*>0$ such that
    \[\lim_{t\to T_*-}\pa_x\rho(t,x)=+\infty,\]
    unless the solution blows up earlier than $T_*$.
  \item If the initial data is type II supercritical, satisfying
    \[\exists~ x_0\in\mathbb{R} \quad\text{s.t.}\quad \rho_0'(x_0)\le\gamma(\rho_0(x_0)),\]
    then the solution must blow up in finite time. More precisely,
    there exists a location $x\in\R$ and a finite time $T_*>0$ such that
   \[\lim_{t\to T_*-}\pa_x\rho(t,x)=-\infty,\]
   unless the solution blows up earlier than $T_*$.
 \end{itemize}
\end{theorem}

\begin{remark}
  The description of the threshold function $\gamma$ is given in
  Theorem \ref{thm:gamma}. It is a function defined for $\rho>\rhoc$
  with a vertical asymptote at $\rho=\rhoc$, namely
  \[\lim_{\rho\to\rhoc+}\gamma(\rho)=-\infty.\]
  The right graph in Figure \ref{fig:CT} illustrates the shapes of the
  threshold functions.
  Our result is sharp: any smooth initial data lie in exactly one of
  the three regions, which then lead to the corresponding global behaviors.

  Unlike the case when $f$ is concave, the threshold functions
  $\sigma$ and $\gamma$ may only be defined in
  a subset of $[0,1]$ and $(\rhoc,1]$ respectively. 
  See Remark \ref{rmk:CTinfinity} for a clarification on the meaning
  of the threshold conditions if $\sigma(\rho_0(x))$ or
  $\gamma(\rho_0(x))$ is undefined.
\end{remark}

Our final result concerns the class of fluxes in
\eqref{eq:fluxfJ}. 
Theorems \ref{thm:concave} and \ref{thm:main} can be applied to the
system with $f=f_J$ for $J\in(0,1]$ and $J>1$, respectively.
Remarkably, we find explicit expressions for the
corresponding threshold functions.

\begin{theorem}
  Suppose the flux $f=f_J$ satisfies \eqref{eq:fluxfJ}. Then the
  threshold functions $\sigma=\sigma_J$ and $\gamma=\gamma_J$ can be
  explicitly expressed as follows.
  For any $J>0$
  \[\sigma_J(\rho) = \frac{\rho(1-\rho)}{J},\quad \rho\in[0,1].\]
  For any $J>1$ we have $\rhoc=\frac{2}{J+1}$ and
  \[\gamma_J(\rho) =
    \frac{\rho^2(1-\rho)\left(\rho-\frac{4J}{(J+1)^2}\right)}{J(\rho-\rhoc)^2},\quad
  \rho\in(\rhoc,1].\]
\end{theorem}

We would like to mention that all our results are based on the
particular choice of kernel in \eqref{eq:ourK}. This allows us to
obtain sharp results.
The kernel $K=1_{[0,\infty)}$ features a jump at the origin,
representing that the interaction is \emph{look-ahead}.
Indeed, such jump drives the main phenomenon: global regularity for a
class of non-trivial subcritical initial data.
We believe the same phenomenon holds for general look-ahead interactions,
where the kernel has the same jump structure at the origin,
like \eqref{eq:K-SK} and \eqref{eq:K-linear}.
We shall leave the generalization for future investigation.

The rest of the paper is organized as follows.
In Section \ref{sec:lwp}, we establish a local wellposedness theory for a
general class of nonlocal traffic flow models, including the system
\eqref{eq:main} of our concern.
In Section \ref{sec:thresholds}, we provide unique constructions of
the threshold functions $\sigma$ and $\gamma$.
In Section \ref{sec:gwp}, we study the global behaviors of solutions
for the three types of initial data, proving Theorems
\ref{thm:concave} and \ref{thm:main}.

\section{Local Wellposedness and regularity criteria}\label{sec:lwp}					
In this section, we establish a local wellposedness theory for a
general class of nonlocal traffic flow models
\begin{equation}\label{eq:general}
  \pa_t\rho+\pa_x\big(f(\rho)e^{-\rhob}\big)=0,\quad \rhob(t,x) =
  \int_\R K(y)\rho(t,x+y)\,dy.
\end{equation}
We shall present the theorem with general assumptions on the kernel
$K$:
\begin{equation}\label{eq:Kgeneral}
  K\in BV(\R),\quad 0\le K(x)\le\KM.
\end{equation}
Here, we only require $K$ to be bounded, nonnegative, and have bounded
total variation. In particular, the interaction does not need to be
look-ahead.
We shall comment that all look-ahead interactions \eqref{eq:K-SK},
\eqref{eq:K-linear} and \eqref{eq:ourK}
satisfy the assumption \eqref{eq:Kgeneral},
with $\KM=1$ and $|K|_{BV}\leq2$.

Let us start with the statement of the local wellposedness theory.
\begin{theorem}[Local wellposedness]\label{thm:lwp}
  Let $k>\frac32$. Consider \eqref{eq:general} with smooth initial condition 
  \[\rho_0 \in L^1_+\cap H^k(\R).\]
 Assume the flux $f$ satisfies \eqref{eq:f}, and the kernel $K$
 satisfies \eqref{eq:Kgeneral}. 
 Then there exists a time $T>0$ such that solution  $\rho=\rho(t,x)$ exists and
 \[\rho\in C\big([0,T];L^1_+\cap H^k(\R)\big).\]
Moreover, the solution exists in $[0,T]$ as long as 
\begin{equation}\label{eq:condition}
\int_0^T\|\partial_x\rho(t,\cdot)\|_{L^\infty}dt <\infty.
\end{equation}
\end{theorem}

Local wellposedness of \eqref{eq:general} has been studied in \cite{lee2019sharp}
for specific flux \eqref{eq:LWR} and interaction kernel \eqref{eq:ourK}.
Here, we extend the result to general fluxes and kernels.
We also provide a regularity criterion \eqref{eq:condition}.
It allows us to study global wellposedness based on the control of $\pa_x\rho$.

In the rest of the section, we present a proof of Theorem
\ref{thm:lwp}, using \emph{a priori} energy estimates. 
The focus is on the proper treatment of the nonlinearity in $f$ and
the nonlocality in the term $e^{-\rhob}$, where nontrivial commutator
and composition estimates are used.

\subsection{A priori bounds}
First, we state the \emph{conservation of mass}. Ingegrating
\eqref{eq:general} in $x$ gives
\[
  \frac{d}{dt}\int_\R\rho(t,x)dx =
  -\int_\R\pa_x\big(f(\rho)e^{-\rhob}\big)dx=0.
\]
Let us denote the total mass
\[
  m\coloneqq\int_\R\rho(t,x)\,dx=\int_\R\rho_0(x)\,dx.
\]

Next, we consider the characterstic path $X(t, x)$ originated at $x\in\R$
\[\pa_tX(t,x) = f'(\rho(t,X(t,x)))e^{-\rhob(t,X(t,x))},\quad X(t=0,x)=x.\]
We shall supress the $x$ dependence and write $X(t)$ from now on.
Along each characterstic path, we have
\begin{equation}\label{eq:rhochar}
  \frac{d}{dt}\rho(t,X(t)) = -\rho(t,X(t)) f(\rho(t,X(t)))e^{-\rhob(t,X(t))}.
\end{equation}
This leads to the following \emph{maximum principle}.

\begin{proposition}[Maximum principle]\label{prop:MP}
  Let $\rhoM\in(0,1]$. Let $\rho=\rho(t,x)$ be a classical solution of
  \eqref{eq:general} in $[0,T]\times\R$ with initial condition
  $\rho_0(x)\in[0,\rhoM]$ for all $x\in\R$.
  Then, $\rho(t,x)\in[0,\rhoM]$ for all $x\in\R$ and $t\in[0,T]$.
\end{proposition}
\begin{proof}
  Since $f(0)=f(1)=0$, $\rho=0$ and $\rho=1$ are equilibrium states of
  \eqref{eq:rhochar}. Hence, $\rho_0\in[0,1]$ implies $\rho(t)\in[0,1]$.
  Moreover, $-\rho f(\rho)e^{-\rhob}<0$ for any $\rho\in(0,1)$. Hence,
  if $\rho_0\leq \rhoM<1$, we have
  $\rho(t)<\rho_0\leq\rhoM$ for any $t\ge0$.
\end{proof}

Finally, we present a priori bounds on the nonlocal term $e^{-\rhob}$.
Applying the definition of $\rhob$ in \eqref{eq:general} and the bounds
on the kernel $K$ in \eqref{eq:Kgeneral}, we obtain the bounds 
\begin{equation}\label{eq:rhobbound}
 0\leq \rhob(t,x)\leq \KM m,
\end{equation}
which then implies
\begin{equation}\label{eq:nonlocalbound}
e^{-\KM m}<e^{-\rhob}\leq 1.
\end{equation}
Furthermore, we have the following bound on $\pa_x(e^{-\rhob})$.

\begin{proposition}\label{prop:nonlocalbound}
  Under the same assumptions as in Proposition \ref{prop:MP}, we have
  \begin{equation}\label{eq:nonlocalboundderi}
    \|\pa_x(e^{-\rhob})\|_{L^\infty}\leq |K|_{BV}.
  \end{equation}
\end{proposition}
\begin{proof}
  First, apply \eqref{eq:nonlocalbound} and get
  \[\|\partial_x(e^{-\rhob})\|_{L^\infty} = \|e^{-\rhob}(-\partial_x\rhob)\|_{L^\infty}
    \le \|\pa_x\rhob\|_{L^\infty}.\]
  It remains to control $\pa_x\rhob$. We apply maximum principle and compute
  \begin{equation}\label{eq:parhobbound}
    |\pa_x\rhob(t,x)|=\left|\int_{-\infty}^\infty K(y)\pa_x\rho(t,x+y)\,dy\right|
    \le |K|_{BV}\cdot \rhoM\le |K|_{BV},
  \end{equation}
  which directly implies \eqref{eq:nonlocalboundderi}.
\end{proof}

\subsection{$L^2$ energy estimate}
Let us integrate \eqref{eq:general} against $\rho$ and get
\begin{align*}
  \frac{1}{2}\frac{d}{dt}\|\rho(t,\cdot)\|_{L^2}^2=&\,
  -\int_\R\rho\,\pa_x\big(f(\rho)e^{-\rhob}\big)\,dx
  = \int_\R\pa_x\rho\, f(\rho)e^{-\rhob}\,dx  =  \int_\R\pa_x F(\rho)\, e^{-\rhob}\,dx\\ 
=&\,  \int_\R- F(\rho)\, \pa_x\big(e^{-\rhob}\big)\,dx
\leq \|\pa_x\big(e^{-\rhob}\big)\|_{L^\infty} \|F(\rho)\|_{L^1}.
\end{align*}
Here the function $F$ is the primitive of $f$. 
From \eqref{eq:f}, we know that $f(x)\leq f'(0)x$ for all $x\in[0,1]$. Therefore, we can
estimate
\[|F(x)|=\left|\int_0^xf(y)\,dy\right|\leq \frac{f'(0)}{2}x^2.\]
Since $\rho\in[0,1]$, we get
\[\|F(\rho)\|_{L^1}\leq \frac{f'(0)}{2}\|\rho\|_{L^2}^2.\]
Apply \eqref{eq:nonlocalboundderi} and we conclude with
\begin{equation}\label{eq:l2est}
  \frac{1}{2}\frac{d}{dt}\|\rho(t,\cdot)\|_{L^2}^2\le \frac{f'(0) |K|_{BV}}{2}\|\rho(t,\cdot)\|_{L^2}^2.
\end{equation}



\subsection{$H^k$ energy estimate} 									

Now, we consider the evolution of the homogeneous $\dot{H}^k$ semi-norm of
$\rho$
\[\|\rho(t,\cdot)\|_{\dot{H}^k}=\|\Lambda^k\rho(t,\cdot)\|_{L^2},\]
where $\Lambda=(-\Delta)^{1/2}$ denotes the fractional Laplacian operator.

Let us first state the following estimates. We refer the proofs to
\cite{lee2019sharp} and references therein.

\begin{lemma}[Fractional Leibniz rule]\label{lem:Leibniz}
Let $k\geq0$. There exists a constant $C>0$, depending only on $k$, such that
$$\|gh\|_{\dot H^k} \leq C\big( \|g\|_{L^\infty} \|h\|_{ \dot H^k} +
\|g\|_{\dot H^k} \|h\|_{L^\infty }\big) .$$
\end{lemma}

\begin{lemma}[Commutator estimate]\label{lem:commutator}
Let $k\geq1$. There exists a constant $C>0$, depending only on $k$, such that
$$\|[\Lambda^k,g]h\|_{L^2} \leq C\big( \|\partial_x g\|_{L^\infty} \|h\|_{ \dot H^{k-1}} + \|g\|_{\dot H^k} \|h\|_{L^\infty }\big) ,$$
where the commutator is denoted by $[\Lambda^k,f]g=\Lambda^k(fg)-f\Lambda^kg$.

\end{lemma}

\begin{lemma}[Composition estimate]\label{lem:composition}
Let $k>0$, $g\in L^{\infty}\cap\dot H^k(\R)$, and $h\in C^{\lceil k \rceil}(\text{Range}(g))$. Then, the composition $h \circ g \in L^{\infty}\cap \dot H^k(\R)$. Moreover, there exists a constant $C>0$, depending on $k$, $\|h\|_{C^{\lceil k \rceil}(\text{Range}(g))}$, and $\|g\|_{L^\infty}$, such that 
$$||h\circ g||_{\dot H^k} \leq C||g||_{\dot H^k}.$$
\end{lemma}

To begin with, we act $\Lambda^k$ on \eqref{eq:general}, integrate
against $\Lambda^k \rho$ and get
\begin{align*}
\frac{1}{2}\frac{d}{dt}\|\rho(t,\cdot)\|_{\dot{H}^k}^2 =&\, -\int_\R \Lambda^k\rho\cdot \Lambda^k \pa_x\big(f(\rho)e^{-\rhob}\big) \,dx\\
=&\, -\int_\R \Lambda^k\rho\cdot \Lambda^k \Big( f'(\rho) \pa_x\rho \cdot e^{-\rhob} - f(\rho)\cdot e^{-\rhob}\pa_x\rhob \Big)\,dx\\
=&\, -\int_\R \Lambda^k\rho\cdot \Lambda^k\pa_x\rho\cdot
   f'(\rho)e^{-\rhob} dx -\int_\R \Lambda^k\rho \cdot [\Lambda^k,f'(\rho)e^{-\rhob}]\partial_x\rho\, dx \\
&+ \int_\R \Lambda^k\rho\cdot\Lambda^k\big( f(\rho)e^{-\rhob}\pa_x\rhob \big)\,dx\\
=&\, {\rm I}+{\rm II}+{\rm III}.
\end{align*}

We bound the three terms one by one. For the first term we use integration by parts
\[
{\rm I}= -\int_\R \pa_x\left(  \frac{(\Lambda^k\rho)^2}{2}\right)
         \cdot f'(\rho)e^{-\rhob} \,dx
 = \frac{1}{2}\int_\R (\Lambda^k\rho)^2\cdot\pa_x\big(f'(\rho)e^{-\rhob}\big) \,dx.\]
Applying \eqref{eq:nonlocalbound} and \eqref{eq:nonlocalboundderi}, we estimate
\[\left|\pa_x\big(f'(\rho)e^{-\rhob}\big)\right|
=\left|f''(\rho)e^{-\rhob}\pa_x\rho + f'(\rho)\pa_x(e^{-\rhob})\right|
\leq \|f\|_{C^2([0,\rhoM])}(\|\pa_x\rho\|_{L^\infty}+|K|_{BV}).\]
This leads to the bound
\begin{equation}\label{eq:I}
  {\rm I}\leq \frac{1}{2}\|f\|_{C^2([0,\rhoM])}(\|\pa_x\rho\|_{L^\infty}+|K|_{BV})\|\rho\|_{\dot{H}^k}^2.
\end{equation}

Moving on to the second term, we apply Lemma \ref{lem:commutator} and get
\begin{align*}
  {\rm II}\leq&\,
  \|\rho\|_{\dot H^k} \big\| [\Lambda^k,f'(\rho)e^{-\rhob}]\partial_x\rho\big\|_{L^2}\\
\leq&\, C(k) \|\rho\|_{\dot H^k} \Big( \|\partial_x( f'(\rho)e^{-\rhob} )\|_{L^\infty} \|\pa_x\rho\|_{ \dot H^{k-1}} + \|f'(\rho)e^{-\rhob}\|_{\dot H^k} \|\pa_x\rho\|_{L^\infty }\Big).
\end{align*}
For convinience in notations, we shall use $C$ to denote the
constants, which can change line by line. We will also keep track of
the dependence of the constant with respect to the parameters.

Now we focus on the estimate of $\|f'(\rho)e^{-\rhob}\|_{\dot H^k}$. Apply Lemma \ref{lem:Leibniz}
\begin{equation}\label{eq:IIP1}
  \|f'(\rho)e^{-\rhob}\|_{\dot H^k}\leq C(k)\Big(\|f'(\rho)\|_{L^\infty} \|e^{-\rhob}\|_{ \dot H^k} + \|f'(\rho)\|_{\dot H^k} \|e^{-\rhob}\|_{L^\infty }\Big).
\end{equation}
The term $\|e^{-\rhob}\|_{ \dot H^k}$ can be estimated as follows.
\begin{proposition}\label{prop:nonlocalHk}
For $k \geq 1$, 
\begin{equation}\label{eq:nonlocalHk}
  \|e^{-\rhob}\|_{\dot H^k} \leq C(k, \KM m, |K|_{BV}) \|\rho\|_{\dot
    H^{k-1}}.
\end{equation}
\end{proposition}
\begin{proof}
We begin by applying Lemma \ref{lem:composition} with
$g(x)=\rhob(t,x)$ and  $h(x)=e^{-x}$. From \eqref{eq:rhobbound} we know
$\|g\|_{L^\infty}\leq \KM m$. Moreover,
$\|h\|_{C^\infty([0, \KM m])}\le 1$. Therefore, we have
\[\|e^{-\rhob}\|_{\dot H^k} \leq C(k, \KM m)\|\rhob\|_{\dot H^k}.\]
Next, we apply Young's inequality and get
\begin{equation}\label{eq:rhobHk}
  \|\rhob\|_{\dot H^k}=\|\pa_x\Lambda^{k-1}\rho\|_{L^2}
  =\left\|\int_\R K(y)\pa_x(\Lambda^{k-1}\rho)\,dy\right\|_{L^2}
  \leq |K|_{BV}\|\Lambda^{k-1}\rho\|_{L^2}.
\end{equation}
Put together and we conclude with \eqref{eq:nonlocalHk}.
\end{proof}

For the term $\|f'(\rho)\|_{\dot{H}^k}$, we again apply Lemma
\ref{lem:composition} with $g(x)=\rho(t,x)$ and $h(x)=f'(x)$.
From the maximum principle, $\|g\|_{L^\infty}\leq \rhoM<1$. Moreover,
\[\|h\|_{ C^{\lceil k \rceil}(\text{Range}(g))}\leq
\|f\|_{ C^{\lceil k \rceil+1}([0,\rhoM])},\]
which is bounded due to the assumptions on $f$ in \eqref{eq:f}.
Hence,
\begin{equation}\label{eq:fprimeHk}
  \|f'(\rho)\|_{\dot{H}^k}\leq
  C(k, \|f\|_{ C^{\lceil k \rceil+1}([0,\rhoM])})\|\rho\|_{\dot H^k}.
\end{equation}

Applying \eqref{eq:nonlocalboundderi}, \eqref{eq:nonlocalHk} and
\eqref{eq:fprimeHk} to \eqref{eq:IIP1} we get
\[\|f'(\rho)e^{-\rhob}\|_{\dot H^k}\leq
C(k, \KM m, |K|_{BV}, \|f\|_{ C^{\lceil k\rceil+1}([0,\rhoM])})\|\rho\|_{H^k}.\]
Consequently, we have the bound on the second term
\begin{equation}\label{eq:II}
{\rm II} \leq C(k, \KM m, |K|_{BV}, \|f\|_{ C^{\lceil k\rceil+1}([0,\rhoM])})
(1+\|\pa_x\rho\|_{L^\infty})\|\rho\|_{\dot H^k}\|\rho\|_{H^k}.
\end{equation}

Finally, let us estimate the third term using Lemma \ref{lem:Leibniz}
\begin{align*}
{\rm III}\leq&\, \|\rho\|_{\dot H^k} \|f(\rho)e^{-\rhob}\pa_x\rhob\|_{\dot H^k}\\
  \leq &\,C(k) \|\rho\|_{\dot H^k} \Big(
  \|f(\rho)\|_{\dot H^k}\|e^{-\rhob}\pa_x\rhob\|_{L^\infty}
  + \|e^{-\rhob}\|_{ \dot H^k} \|f(\rho) \pa_x\rhob\|_{L^\infty}\Big.\\
  &\hspace{6pc}\Big.+\|\pa_x\rhob\|_{\dot H^k}\|f(\rho)e^{-\rhob}\|_{L^\infty }\Big)\\
= &\,C(k) \|\rho\|_{\dot H^k} \big({\rm III}_1+{\rm III}_2+{\rm III}_3\big).
\end{align*}
For ${\rm III}_1$, use \eqref{eq:nonlocalbound},
\eqref{eq:parhobbound} and \eqref{eq:fprimeHk} (with $f'$ replaced by $f$)
\[{\rm III}_1\leq C(k, \|f\|_{ C^{\lceil k\rceil}([0,\rhoM])}) |K|_{BV}\|\rho\|_{\dot H^k}.\]
For ${\rm III}_2$, use \eqref{eq:parhobbound} and \eqref{eq:nonlocalHk}
\[{\rm III}_2\leq C(k, \KM m, |K|_{BV}) \|f\|_{ C^0([0,\rhoM])}|K|_{BV} \|\rho\|_{\dot H^{k-1}}.\]
For ${\rm III}_3$, use \eqref{eq:nonlocalbound} and \eqref{eq:rhobHk} 
\[{\rm III}_3\leq \|f\|_{ C^0([0,\rhoM])}|K|_{BV} \|\rho\|_{\dot H^k}.\]
All together, we obtain
\begin{equation}\label{eq:III}
  {\rm III}\leq C(k, \KM m, |K|_{BV}, \|f\|_{ C^{\lceil k\rceil}([0,\rhoM])})
\|\rho\|_{\dot H^k}\|\rho\|_{H^k}.
\end{equation}

Collecting the estimates \eqref{eq:I}, \eqref{eq:II}, \eqref{eq:III},
we end up with the estimate on $H^k$ energy $(k\geq1)$ as follows.
\begin{equation}\label{eq:Hkest}
  \frac{1}{2}\frac{d}{dt}\|\rho(t,\cdot)\|_{\dot H^k}^2
  \leq C(k, \KM m, |K|_{BV}, \|f\|_{ C^{\lceil k\rceil+1}([0,\rhoM])})
  (1+\|\pa_x\rho\|_{L^\infty}) \|\rho\|_{\dot H^k}\|\rho\|_{H^k},
\end{equation}
where the constant $C$ is finite under our assumptions on $f$ and $K$.

\subsection{Proof of Theorem \ref{thm:lwp}}
Define an energy
\[Y(t)=\|\rho(t,\cdot)\|_{L^2}^2+\|\rho(t,\cdot)\|_{\dot H^k}^2.\]
Clearly, $Y(t)$ is equivalent to $\|\rho(t,\cdot)\|_{H^k}^2$.
Combining the $L^2$ and $H^k$ energy estimates \eqref{eq:l2est} and
\eqref{eq:Hkest}, we have the bound on the evolution of $Y$ as follows
\begin{equation}\label{eq:energybound}
Y'(t)\leq C(1+\|\pa_x\rho(t,\cdot)\|_{L^\infty})\|\rho(t,\cdot)\|_{H^k}^2.
\end{equation}
Let $k>3/2$, from the Sobolev embedding theorem, we have
$\|\pa_x\rho\|_{L^\infty}\leq C(k)\|\rho\|_{H^k}$. This leads to a
bound
\[Y'(t)\leq C(1+Y^{1/2})Y(t).\]
Clearly, there exists a time $T_*>0$, depending on $Y(0)$ and $C$,
such that $Y(t)$ exists and is bounded for $t\in[0,T_*]$. This
finishes the local wellposedness proof.

Moreover, we apply Gr\"onwall inequality to \eqref{eq:energybound}
and obtain
\[Y(T)\leq Y(0)\exp\left(\int_0^TC(1+\|\pa_x\rho(t,\cdot)\|_{L^\infty})\,dt\right).\]
Therefore, $Y(T)$ remains bounded if criterion \eqref{eq:condition} holds.

\section{Critical Thresholds}\label{sec:thresholds}
In this section, we restrict our attention to our main equation
\eqref{eq:main} with the special kernel \eqref{eq:ourK}.
The goal is to construct threshold functions that distinguish the
global behaviors of the solutions.

From the regularity criterion \eqref{eq:condition}, we know that
the solution is globally regular if and only if $\pa_x\rho$ is
bounded.
Let us denote
\[d=\partial_x\rho.\]
We shall focus on the boundedness of $d$.

Differentiating \eqref{eq:main} in $x$, we can write the dynamics
of $d$ as
\[\pa_td+f'(\rho)e^{-\rhob}\pa_xd=\big(-f''(\rho)d^2-(f(\rho)+2\rho f'(\rho))d-\rho^2f(\rho)\big)e^{-\rhob}.\]
Here we have used the special structure of \eqref{eq:rhobar}. In particular,
\[\pa_x\rhob=-\rho.\]

Let us denote $\dot{d}$ as the time derivative along the characteristic a
path $X(t)$, namely
\[\dot{d}=\frac{d}{dt} d(t,X(t)).\]
Then, together with \eqref{eq:rhochar}, we obtain a coupled dynamics of
$(\rho,d)$ along each characterstic path
\begin{equation}\label{eq:drho}
  \begin{cases}
    \,\dot{\rho} = -\rho f(\rho) e^{-\rhob},\\
    \,\dot{d}=-\big(f''(\rho)d^2+(f(\rho)+2\rho f'(\rho))d+\rho^2f(\rho)\big)e^{-\rhob}.
  \end{cases}
\end{equation}
Note that the only nonlocality in the coupled dynamics \eqref{eq:drho}
appears to be the factor $e^{-\rhob}$. Thus, the trajectories on the
phase plane $(\rho,d)$ depend on the local information. Indeed, if we
express a trajectory as $d=\d(\rho)$, then it satisfies the following
differential equation
\begin{equation}\label{eq:trajode}
\d'(\rho) = \frac{f''(\rho)\d^2+(f(\rho)+2\rho f'(\rho))\d+\rho^2f(\rho)}{\rho f(\rho)}.
\end{equation} 

We will examine the trajectories in the phase plane and investigate
whether the trajectories are bounded or not. The boundedness of $\d$
will then lead to global wellposedness of the system \eqref{eq:main}
by Theorem \ref{thm:lwp}.

There are two special trajectories that serve as thresholds in the
phase plane. They divide the area $\{(\rho,d) : \rho\in[0,1]\}$ into
three regions. Trajectories originated on each region lead to different large
time behaviors.
We call the two trajectories \emph{critical threshold functions},
and denote
them by two functions $\sigma$ and $\gamma$. The trajectories are
expressed as $d=\sigma(\rho)$ and $d=\gamma(\rho)$ respectively.
Figure \ref{fig:CT} illustrates the shapes of the threshold functions.

In the following, we focus on the wellposedness of the two critical
thresholds.

\subsection{The threshold function $\sigma$}
The curve $\sigma$ represents a trajectory that goes across $(0,0)$ in the
phase plane. Since $(0,0)$ is a degenerate equilibrium state of the
phase dynamics, there are infinitely many trajectories such that
$\d(0)=0$. These trajectories satisfy the following property.
\begin{proposition}\label{prop:initiald}
  Let $d=\d(\rho)$ be a trajectory such that $\d(0)=0$. Assume $\d'(0)$
  exists. Then, we must have
  \begin{equation}\label{eq:dprime0}
    \d'(0)=0\quad\text{or}\quad \d'(0)=-\frac{2f'(0)}{f''(0)}.
  \end{equation}
\end{proposition}
\begin{proof}
  We apply \eqref{eq:trajode} and take $\rho\to0$
\begin{align*}
\d'(0) &= \lim_{\rho\to 0_+} \frac{f''(\rho)\d(\rho)^2+(f(\rho)+2 \rho f'(\rho))\d(\rho)+\rho^2f(\rho) }{\rho f(\rho)}\\
&=\lim_{\rho\to 0+} \frac{\rho f''(\rho)}{f(\rho)}\cdot \d'(0)^2+
\lim_{\rho\to 0+} \frac{2 \rho f'(\rho)+f(\rho)}{f(\rho)}\cdot \d'(0)
= \frac{f''(0)}{f'(0)}\cdot \d'(0)^2+3\d'(0).
\end{align*}
This directly leads to \eqref{eq:dprime0}.
\end{proof}

To simplify the notation, we denote
\[\beta = -\frac{2f'(0)}{f''(0)}\]
for the rest of the section. Note that $\beta>0$.

Among these trajectories, there is only one such that
$\d'(0)=\beta$. This is the trajectory $\sigma$ that
we seek for. The following theorem ensures a uniquely defined
threshold curve $\sigma$. The idea of the proof follows from
\cite[Proposition 3.1]{lee2019sharp}.

\begin{theorem}\label{thm:sigma}
There exists a unique trajectory represented by $\sigma$ that satisfies the equation
\eqref{eq:trajode}, namely
\begin{subequations}\label{eqs:sigmaode}
\begin{equation}\label{eq:sigmaode}
     \sigma'(\rho) = \frac{f''(\rho)\sigma(\rho)^2+(f(\rho)+2\rho f'(\rho))\sigma(\rho)+\rho^2f(\rho)}{\rho f(\rho)},
\end{equation}
with initial conditions
\begin{equation}\label{eq:sigmaodeinit}
  \sigma(0) = 0,\quad\text{and}\quad  \sigma'(0) = \beta.
\end{equation}
\end{subequations}
\end{theorem}
\begin{proof}
  We start with the local existence theory. Fix a small $\epsilon>0$.
  The classical Cauchy-Peano theorem does
  not apply directly near $x=0$, as the right hand-side of \eqref{eq:sigmaode}
  \[F(\rho,\sigma):=\frac{f''(\rho)\sigma^2+(f(\rho)+2\rho f'(\rho))\sigma
      +\rho^2f(\rho)}{\rho f(\rho)}\]
  is not uniformly bounded for
  $(\rho,\sigma)\in[0,\epsilon]\times[-\epsilon, \epsilon]$.
  By smallness of $\epsilon$ and smoothness of $f$, we have
  \[F(\rho,\sigma) = -\frac{2}{\beta}\left(\frac{\sigma}{\rho}\right)^2+3 \left(\frac{\sigma}{\rho}\right)+\mathcal{O}(\epsilon),\]
  for any $\rho$ inside the region
  \[A=\left\{(\rho,\sigma): 0\leq\sigma\leq \frac{5\beta}{4} \rho,~~0\leq
      \rho\leq\epsilon\right\}.\]
 We can check that if $\frac\sigma \rho\in[0,\frac{5\beta}{4}]$, then
  \[ \min\left\{\rho,\,\frac{5\beta}{8}+\mathcal{O}(\epsilon)\right\}\leq F(\rho,\sigma)\leq
    \frac{9\beta}{8}+\mathcal{O}(\epsilon).\]
  Hence, if we pick $\epsilon$ small enough, we would have
  \begin{equation}\label{eq:Fxsigmabound}
    0\leq F(\rho,\sigma)\leq \frac{5\beta}{4},\quad \forall~(\rho,\sigma)\in A.
  \end{equation}
Now, we can build a sequence of approximate solutions $\{\sigma_n(\rho)\}$ for 
$\rho\in[0,\epsilon]$.
Given $n\in\mathbb{Z}_+$, define equi-distance lattice $\{\rho_k=\frac{k\epsilon}{n}\}_{k=0}^n$.
\begin{align*}
  \text{(i).}\quad&\sigma_n(\rho)=\beta \rho,\quad \forall~\rho\in\left[0,\rho_1\right].\\
  \text{(ii).}\quad&\sigma_n(\rho)=\sigma_n(\rho_k)+
   F(\rho_k,\sigma_n(\rho_k))(\rho-\rho_k),\quad \forall~\rho\in[\rho_k,\rho_{k+1}],\quad k=1,\cdots,n-1.
\end{align*}
From \eqref{eq:Fxsigmabound}, we know $(\rho,\sigma_n(\rho))\in A$, for all
$\rho\in[0,\epsilon]$.
Hence, $\sigma_n(\rho)$ is uniformly bounded and equi-continuous in
$\rho\in[0,\epsilon]$. By the Arzela-Ascoli theorem, $\sigma_n$ converges
uniformly to $\sigma$, up to an extraction of a subsequence. And
by its construction, $\sigma$ is indeed a solution of \eqref{eq:sigmaode}.

Next, we verify the initial conditions \eqref{eq:sigmaodeinit}.
It is clear that $\sigma(0)=0$ since $\sigma_n(0)=0$ for every $n$. To
verify $\sigma'(0)=\beta$, we show the following statement: the image of
the solution $(\rho,\sigma(\rho))$ lies inside the cone
  \[\{(\rho,\sigma): (1-\delta)\beta \rho\leq\sigma\leq (1+\delta)\beta \rho,~~0\leq \rho\leq\epsilon\}.\]
Indeed, we check $F$ at the boundary of the cone
\begin{align}
  F\big(\rho,\sigma=(1-\delta)\beta \rho\big)=&~\beta(1+\delta-2\delta^2)+\mathcal{O}(\epsilon)>\beta,\label{eq:Fconelb}\\
 F\big(\rho,\sigma=(1+\delta)\beta \rho\big)=&~\beta(1-\delta-2\delta^2)+\mathcal{O}(\epsilon)<\beta,\nonumber
\end{align}
where the inequalities can be obtained by choosing
$\delta=\sqrt{\epsilon}$ and let $\epsilon$ small enough.
Therefore, $\sigma'(\rho)\in[(1-\delta)\beta, (1+\delta)\beta]$ for all
$\rho\in[0,\epsilon]$. Take $\epsilon\to0$, we conclude with $\sigma'(0)=\beta$.

Finally, we discuss the local uniqueness. Let $\sigma^{(1)}$ and
$\sigma^{(2)}$ be two different solutions of \eqref{eqs:sigmaode}.
Fix a small $\epsilon$. From \eqref{eq:Fconelb} we know that
$\sigma^{(i)}(\rho)\geq (1-\delta)\beta \rho$ for $i=1,2$.
Let $w=\sigma^{(1)}-\sigma^{(2)}$.
Note that $\sigma^{(1)}$ and $\sigma^{(2)}$ can not cross each other for $\rho\in(0,\epsilon]$.
Without loss of generality, we may assume $w(\rho)>0$ for
$\rho\in(0,\epsilon]$. (Otherwise, switch $\sigma^{(1)}$ and $\sigma^{(2)}$).
Compute
\begin{align*}
  w'(\rho)=&
          ~\frac{f''(\rho)(\sigma^{(1)}(\rho)+\sigma^{(2)}(\rho))+(f(\rho)+2\rho f'(\rho))}{\rho f(\rho)}w(\rho)\\
  \leq&~ \frac{f''(0)\cdot 2(1-\delta)\beta
        \rho+3f'(0)\rho+\mathcal{O}(\rho^2)}{\rho f(\rho)}w(\rho)=\frac{-f'(0) \rho+\mathcal{O}(\delta
        \rho)}{\rho f(\rho)}w(\rho)<0,
\end{align*}
for any $\rho\in(0,\epsilon]$.
Since $w(0)=0$, it implies $w(\rho)<0$ for $\rho\in(0,\epsilon]$.
This leads to a contradiction.

Once we obtain local wellposedness of $\sigma$ near $\rho=0$, global existence and
uniqueness for $\rho>0$ follows from the standard Cauchy-Lipschitz
theorem. Indeed, $F(\rho,\sigma)$ is bounded and Lipschitz in $\sigma$ as
long as $\rho\in(0,1)$ and $\sigma$ is bounded.
\end{proof}

Next, we discuss properties of the threshold function $\sigma$.

\begin{proposition}\label{prop:sigmapos}
 For any $\rho\in(0,1)$ that lies in the domain of $\sigma$, $\sigma(\rho)>0$.
\end{proposition}
\begin{proof}
  Suppose the argument is false. Then there must exist 
  \[\rho_z=\min \{\rho\in(0,1) : \sigma(\rho)=0\}>0,\]
  such that $\sigma(\rho_z)$ returns to zero for the first time.
  Clearly, $\sigma(\rho)>0$ for all $\rho\in(0,\rho_z)$.
  This implies $\sigma'(\rho_z)\le0$. On the other hand, from the
  dynamics \eqref{eq:sigmaode} and $\sigma(\rho_z)=0$ we have
  $\sigma'(\rho_z)=\rho_z>0$. This leads to a contradiction.
\end{proof}
The positivity of $\sigma$ allows the subcritical regions in Theorems
\ref{thm:concave} and \ref{thm:main} to contain initial data $\rho_0$
that is not monotone decreasing. It is a major indication that the
nonlocal slowdown interaction helps to prevent shock formations for a
class of non-trivial initial data.

Generally speaking, it is possible that $\sigma$ can become
unbounded. The following Proposition describes the behaviors of $\sigma$.

\begin{proposition}\label{prop:sigmablow}
  Let $\sigma$ be the solution of \eqref{eqs:sigmaode}. Then
  exactly one of the following statement is true.
  \begin{itemize}
    \item $\sigma$ is well-defined in $[0,1]$.
    \item There exists a $\rho_*\in(0,1]$ such that
      \begin{equation}\label{eq:sigmablowup}
        \lim_{\rho\to \rho_*-}\sigma(\rho)=+\infty.
      \end{equation}
      Moreover, we have $\rho_*>\rhoc$.
  \end{itemize}
\end{proposition}
\begin{proof}
Suppose \eqref{eq:sigmablowup} does not hold, namely $\sigma$ is
bounded from above in $[0,1]$. Together with Proposition
\eqref{prop:sigmapos}, we know $\sigma$ is bounded. Hence, Theorem
\ref{thm:sigma} implies the existence and uniqueness of $\sigma$ in
$[0,1]$.

We are left to show that $\rho_*>\rhoc$, namely blowup cannot happen
before $\rhoc$. To this end, we observe that $f''(\rho)\leq0$ for all
$\rho\in[0,\rhoc]$. We can estimate from \eqref{eq:sigmaode} that
\[\sigma'(\rho)\leq 0\cdot\sigma(\rho)^2+M\sigma(\rho)+1,\quad
  \text{where}\quad
  M=\max_{\rho\in[\epsilon,\rhoc]}\frac{f(\rho)+2\rho f'(\rho)}{\rho
    f(\rho)}<+\infty,\]
for any $\rho\in[\epsilon,\rhoc]$. This implies the upper bound
\[\sigma(\rho)\leq\left(\sigma(\epsilon)+\frac{1}{M}\right)e^{M\rho},\quad
  \forall~\rho\in[\epsilon,\rhoc].\]
Therefore, the blowup cannot happen when $\rho\leq\rhoc$.
\end{proof}

When the flux $f$ is concave, namely $\rhoc=1$, the second statement
in Proposition \ref{prop:sigmablow} won't hold. Hence, $\sigma$ is
well-defined in $[0,1]$.
When $f$ switches from concave to convex at $\rhoc<1$, one can not
guarantee that $\sigma$ won't blow up. 
However, for the particular $f_J$ in \eqref{eq:fluxfJ} of our concern, $\sigma_J$ is
well-defined in $[0,1]$, even if $J>1$.
Moreover, we find the explicit expression of the threshold function $\sigma_J$.
\begin{proposition}
  Let $f(\rho)=f_J(\rho)=\rho(1-\rho)^J$ for $J>0$. Then the
  trajectory $\sigma_J$ in \eqref{eq:sigmaode} can be explicitly
  expressed by
  \begin{equation}\label{eq:sigmasol}
    \sigma_J(\rho) = \frac{\rho(1-\rho)}{J}.
  \end{equation}
\end{proposition}
\begin{proof}
  We verify that $\sigma_J$ solves \eqref{eqs:sigmaode}.
  For equation \eqref{eq:sigmaode}, we plug in $f_J$ and $\sigma_J$ to the right
  hand side and get
  \begin{align*}
    &\,\frac{f_J''(\rho)\sigma_J(\rho)^2+(f_J(\rho)+2\rho f_J'(\rho))\sigma_J(\rho)+\rho^2f_J(\rho)}{\rho f_J(\rho)}\\
    =&\,\frac{J(1-\rho)^{J-2}(-2(1-\rho)+(J-1)\rho)\cdot
       J^{-2}\rho^2(1-\rho)^2}{\rho^2(1-\rho)^J}\\
      &\,+\frac{\rho(1-\rho)^{J-1}(3(1-\rho)-2J\rho)\cdot J^{-1}\rho(1-\rho)}{\rho^2(1-\rho)^J}+\rho\\
    =&\,\frac{-2(1-\rho)+(J-1)\rho}{J}+\frac{3(1-\rho)-2J\rho}{J}+\rho=\frac{1-2\rho}{J}=\sigma_J'(\rho).
  \end{align*}
  For the initial conditions \eqref{eq:sigmaodeinit}, we can easily
  verify that $\sigma_J(0)=0$, $\sigma_J'(0)=\frac1J$ and $\beta_J=-\frac{2f_J'(0)}{f_J''(0)}=\frac1J$.
\end{proof}

\subsection{The threshold function $\gamma$}
Next, we describe the construction of the other threshold function
$\gamma$, when the flux switches from concave to convex at $\rhoc<1$.

\begin{theorem}\label{thm:gamma}
There exists a unique trajectory represented by $\gamma$ that satisfies the equation
\eqref{eq:trajode}, namely
\begin{subequations}\label{eqs:gammaode}
\begin{equation}\label{eq:gammaode}
     \gamma'(\rho) = \frac{f''(\rho)\gamma(\rho)^2+(f(\rho)+2\rho f'(\rho))\gamma(\rho)+\rho^2f(\rho)}{\rho f(\rho)},
\end{equation}
with initial condition
\begin{equation}\label{eq:gammaodeinit}
  \lim_{\rho\to\rhoc+}\gamma(\rho) =-\infty. 
\end{equation}
\end{subequations}
\end{theorem}
\begin{proof}
  Let us first construct $\gamma$ locally in $(\rhoc, \rhoc+\epsilon)$, for a
  sufficiently small $\epsilon>0$. $\gamma$ can be defined via
  $\eta=\frac{1}{\gamma}$. Indeed, as $\gamma$ satisfies
  \eqref{eq:gammaode}, we must have
  \begin{subequations}\label{eqs:etaode}
  \begin{equation}\label{eq:etaode}
    \eta'(\rho)=-\frac{\gamma'(\rho)}{\gamma(\rho)^2}=\frac{-f''(\rho)-(f(\rho)+2\rho
      f'(\rho))\eta(\rho)-\rho^2f(\rho)\eta(\rho)^2}{\rho f(\rho)}.
  \end{equation}
  Then, equation \eqref{eq:etaode} with initial condition
  \begin{equation}\label{eq:etaodeinit}
    \eta(\rhoc)=0
\end{equation}
\end{subequations}
is locally wellposed in $[\rhoc,\rhoc+\epsilon]$.
We claim that $\gamma(\rho)=\frac{1}{\eta(\rho)}$ satisfies
\eqref{eqs:gammaode} for $\rho\in(\rhoc,\rhoc+\epsilon]$.
It suffices to show that $\eta(\rhoc+)<0$. To this end, take Taylor
expansion of $\eta$ around $\rhoc$
\[\eta(\rho)=\sum_{n=0}^\infty\frac{\eta^{(n)}(\rhoc)}{n!}(\rho-\rhoc)^n.\]
The first term of the series is zero due to the initial condition
\eqref{eq:etaodeinit}.
For the second term, observe from the assumption of $f$ in
\eqref{eq:f} that $f''(\rhoc)=0$. Then from \eqref{eqs:etaode} we get
$\eta'(\rhoc)=0$.
We continue to calculate the next term
\[\eta''(\rhoc)=-\frac{f'''(\rhoc)}{\rhoc f(\rhoc)}.\]
Since $f$ switches from concave to convex at $\rho=\rhoc$, we have
$f'''(\rhoc)\ge0$. If the strict inequality holds, we have
$\eta''(\rhoc)<0$, which yields $\eta(\rhoc+)<0$. If $f'''(\rhoc)=0$,
we can continue to the next terms in the Taylor expansion until we
have $f^{(n)}(\rhoc)>0$ for some $n$.
Note that such finite $n$ exists, as otherwise $f$ is linear around
$\rhoc$, violating the strict convexity assumption in \eqref{eq:f}.
Then
\[\eta(\rhoc)=\eta'(\rhoc)=\cdots=\eta^{(n-2)}(\rhoc)=0,\quad
  \eta^{(n-1)}(\rhoc)=-\frac{f^{(n)}(\rhoc)}{\rhoc f(\rhoc)}<0,\]
which also leads to $\eta(\rhoc+)<0$.
Note that the $\eta$ defined in \eqref{eqs:etaode} is unique. Hence,
$\gamma$ can also be uniquely defined in $(\rhoc, \rhoc+\epsilon]$ by
$\gamma(\rho)=\frac{1}{\eta(\rho)}$.

Starting from $\rho=\rhoc+\epsilon$, we can consider the dynamics
\eqref{eq:gammaode} with initial condition $\gamma(\rhoc+\epsilon)=
\frac{1}{\eta(\rhoc+\epsilon)}$. From the Cauchy-Lipschitz theorem, 
$\gamma$ exists and is unique in $\rho\in[\rhoc+\epsilon, 1)$ as long
as $\gamma$ is bounded.

We now show $\gamma$ is lower bounded when $\rho\ge\rhoc+\epsilon$.
Let us start with an estimate on $f(\rho)+2\rho f'(\rho)$. Applying
convexity of $f$, we have
\[f(\rho)=f(1)-f'(\rho)(1-\rho)-\frac{f''(\xi)}{2}(1-\rho)^2\le -f'(\rho)(1-\rho).\]
Here $\xi\in[\rho,1]$ and $f''(\xi)\ge0$.
Since $f'(\rho)<0$ for $\rho\in(\rhoc, 1)$, we obtain
\[f(\rho)+2\rho f'(\rho)\leq f'(\rho)(3\rho-1)<0,\quad\text{if}\quad \rho>\tfrac13.\]
Then from \eqref{eq:gammaode} it is easy to verify that
$\gamma'(\rho)>0$ as long as $\gamma(\rho)\le0$.
Therefore, it is not possible that $\gamma(\rho)\to-\infty$ for any
$\rho>\max\{\rho_c+\epsilon,\frac13\}$.
It remains to show that $\gamma(\rho)$ is lower bounded when
$\rho\in[\rhoc+\epsilon, \frac13]$, in the case
$\rhoc+\epsilon<\frac13$.
From strict convexity in \eqref{eq:f}, we obtain a uniform bound
\[\sup_{\rho\in\left[\rhoc+\epsilon, \,\frac13\right]}\frac{f(\rho)}{f''(\rho)}< M,\]
where $M$ depends on $\epsilon$. Now, if $\gamma(\rho)<-M$ we apply
\eqref{eq:gammaode} and get 
\[\gamma'(\rho)=\frac{f''(\rho)\gamma(\rho)+f(\rho)}{\rho f(\rho)}\gamma(\rho)+
\frac{2f'(\rho)\gamma(\rho)}{f(\rho)}+\rho>0,\]
as all three terms above are positive. We conclude with a lower bound
of $\gamma$
\[\gamma(\rho)\geq \min\{-M, \gamma(\rhoc+\epsilon)\}.\]

For the upper bound, since $\gamma$ and $\sigma$ are two trajectories
satisfying the same ODE, we have the bound
\[\gamma(\rho)<\sigma(\rho).\]
Therefore, if $\sigma$ is bounded, so is $\gamma$. On the other
hand, if $\sigma$ becomes unbounded as in \eqref{eq:sigmablowup}, it
is possible that $\gamma$ also becomes unbounded, namely there exists
$\rho^*\in(\rho_*,1)$ 
\begin{equation}\label{eq:gammablowup}
  \lim_{\rho\to \rho^*-}\gamma(\rho)=+\infty.
\end{equation}
\end{proof}

\begin{proposition}
  Let $f(\rho)=f_J(\rho)=\rho(1-\rho)^J$ for $J>1$. Then the
  trajectory $\gamma_J$ in \eqref{eq:sigmaode} can be explicitly
  expressed by
  \begin{equation}\label{eq:sigmasol}
    \gamma_J(\rho) = \frac{\rho^2(1-\rho)\left(\rho-\frac{4J}{(J+1)^2}\right)}{J\left(\rho-\frac{2}{J+1}\right)^2}.
  \end{equation}
\end{proposition}
\begin{proof}
  First, we calculate
  \[f''(\rho) = J\big((J+1)\rho-2\big)(1-\rho)^{J-2},\]
  and hence the inflection point $\rhoc=\frac{2}{J+1}$.
  Let us denote $\rho_e=\frac{4J}{(J+1)^2}$.
  Since $\rhoc<\rho_e$, it is easy to verify the
  condition \eqref{eq:gammaodeinit}.
  Now, we verify the equation \eqref{eq:gammaode}.
  Differentiate \eqref{eq:sigmasol} and get
  \begin{equation}\label{eq:gammaJprime}
    \gamma_J'(\rho)=\frac{\rho}{J(\rho-\rhoc)^3}\cdot\Big(
-2\rho^3+(\rho_e+4\rhoc+1)\rho^2-3(\rho_e+1)\rhoc \rho+2\rho_e\rhoc\Big).
\end{equation}
  Plug in $f_J$ and
  $\gamma_J$ to the right hand side of \eqref{eq:gammaode} and get
  \begin{align*}
    &\,\frac{f_J''(\rho)\gamma_J(\rho)^2+(f_J(\rho)+2\rho f_J'(\rho))\gamma_J(\rho)+\rho^2f_J(\rho)}{\rho f_J(\rho)}\\
    =&\,J(J+1)(\rho-\rhoc)(1-\rho)^{J-2}\cdot\frac{\rho^4(1-\rho)^2(\rho-\rho_e)^2}{J^2(\rho-\rhoc)^4}\cdot
    \frac{1}{\rho^2(1-\rho)^J}\\
      &\,+\rho(1-\rho)^{J-1}(3(1-\rho)-2J\rho)\cdot\frac{\rho^2(1-\rho)(\rho-\rho_e)}{J(\rho-\rhoc)^2}\cdot
    \frac{1}{\rho^2(1-\rho)^J}+\rho\\
    =&\,\frac{(J+1)\rho^2(\rho-\rho_e)^2+\rho(3(1-\rho)-2J\rho)(\rho-\rho_e)(\rho-\rhoc)+J\rho(\rho-\rhoc)^3}{J(\rho-\rhoc)^3}\\
    =&\,\frac{\rho}{J(\rho-\rhoc)^3}\cdot\Big[
       -2\rho^3+(\rho_e+(3-J)\rhoc+3)\rho^2\Big.\\
    &\qquad\qquad\Big.
      +\big(\rho_e((J+1)\rho_e-2J\rhoc)-3(\rho_e\rhoc+\rhoc+\rho_e-J\rhoc^2)\big)\rho
      +(3\rho_e-J\rhoc^2)\rhoc
       \Big].  \end{align*}
 Using the definitions of $\rhoc$ and $\rho_e$, we see that the
 expression matches with \eqref{eq:gammaJprime}. We conclude that $\gamma_J$
 satisfies \eqref{eq:gammaode}.
\end{proof}

\section{Global behaviors of solutions}\label{sec:gwp}

In this section, we study the dynamics \eqref{eq:drho} with initial
conditions
\begin{equation}\label{eq:drhoinit}
  \rho(t=0)=\rho_0\in[0,\rhoM], \quad d(t=0)=d_0.
\end{equation}
We argue that when $(\rho_0,d_0)$ lie in different regions in the
phase plane separated by the threshold functions $\sigma$ and
$\gamma$, the global behaviors of the dynamics vary.

\begin{theorem}\label{thm:drho}
  Consider the system \eqref{eq:drho} with initial data $(\rho_0,d_0)$
  as in \eqref{eq:drhoinit}. Then,
  \begin{itemize}
    \item[(a).] If $(\rho_0, d_0)$ lies in the type I supercritical region,
      that is
      \begin{equation}\label{eq:super1}
        d_0>\sigma(\rho_0),
      \end{equation} 
      then there exists a finite time $t_*$, such that
      \[\lim_{t\to t_*-}d(t)=+\infty.\]
    \item[(b).] If $(\rho_0, d_0)$ lies in the type II supercritical region,
      that is
      \begin{equation}\label{eq:super2}
        \rho_0>\rhoc \quad\text{and}\quad d_0\le\gamma(\rho_0),
      \end{equation}
      then there exists a finite time $t_*$, such that
      \[\lim_{t\to t_*-}d(t)=-\infty.\]
    \item[(c).] If $(\rho_0, d_0)$ lies in the subcritical region, meaning
      neither of the two supercritical regions, that is
      \begin{equation}\label{eq:sub1}
        \rho_0\le\rhoc\quad\text{and}\quad d_0\le\sigma(\rho_0),
      \end{equation}
      or
      \begin{equation}\label{eq:sub2}
        \rho_0>\rhoc\quad\text{and}\quad \gamma(\rho_0)<d_0\le\sigma(\rho_0),
      \end{equation}
      then the solution $(\rho,d)$ exists in all
      time. Moreover,
      \[\lim_{t\to\infty}\rho(t)=0,\quad \lim_{t\to\infty}d(t)=0.\]
  \end{itemize}
\end{theorem}
\begin{remark}\label{rmk:CTinfinity}
  Since $\sigma$ and $\gamma$ might not be defined in $[0,1]$ and
  $(\rhoc,1]$ respectively,
  in the cases \eqref{eq:sigmablowup} and \eqref{eq:gammablowup} when
  blowups happen, we shall adapt the following conventions in
  the descriptions of the regions \eqref{eq:super1}, \eqref{eq:super2}
  and \eqref{eq:sub2}:
  \[\sigma(\rho)=+\infty,\quad\forall~\rho\in[\rho_*,1],
    \quad\text{and}\quad
    \gamma(\rho)=+\infty,\quad\forall~\rho\in[\rho^*,1].\]
  For instance, if $\rho_0\ge\rho_*$, \eqref{eq:super1} implies
  that $(d_0,\rho_0)$ does not lie in the type I supercritical region
  for any $d_0\in\R$.
\end{remark}

Figure \ref{fig:CT} provides an illustration of the three regions.
Note that if $f$ is concave, satisfying \eqref{eq:f} with $\rhoc=1$,
the type II supercritical region \eqref{eq:super2} is empty.
The region \eqref{eq:sub2} is also empty.

Theorems \ref{thm:concave} and \ref{thm:main} follow directly from
Theorem \ref{thm:drho} by collecting all the characteristic paths.

The rest of the section is devoted to the proof of Theorem
\ref{thm:drho}. We start with a description of the dynamics of
$\rho$: it decreases in time, and approaches zero as $t\to\infty$.

\begin{lemma}\label{lem:rhodown}
  Consider \eqref{eq:drho} with initial data \eqref{eq:drhoinit}.
  Then $\rho(t)$ is strictly decreasing in time.
  Moreover, for any $\rho_1\in(0,\rho_0)$, there exists a finite time
  $t_1$ such that $\rho(t_1)\le\rho_1$, unless $d(t)$ blows up before
  $t_1$.
  In particular, if $(\rho,d)$ exists in all time, then
  \[\lim_{t\to\infty}\rho(t)=0.\]
\end{lemma}
\begin{proof}
  Apply the bound \eqref{eq:nonlocalbound} to the $\rho$-equation in
  \eqref{eq:drho} and get
  $\dot\rho \geq - \rho f(\rho)$, which implies
  $\rho(t)\geq \rho_0e^{-\|f\|_{L^\infty}t}>0$ is strictly positive in
  all time. On the other hand, we have 
  \[\dot\rho \leq -e^{-m} \rho f(\rho).\]
  Since the right hand side is strictly negative when $\rho\in(0,1)$,
  we conclude that $\rho(t)$ is strictly decreasing in time.
  
  Moreover, using separation of variables and integrating in $[0,t]$ yield
  \begin{equation}\label{eq:rhoint}
    \int_{\rho_0}^{\rho(t)}\frac{d\rho}{\rho f(\rho)}\leq -e^{-m}t.
  \end{equation}
  Define $G: (0,\rho_0]\to(-\infty,0]$ as follows
  \[G(\xi):=\int_{\rho_0}^{\xi}\frac{d\rho}{\rho f(\rho)}.\]
  Clearly, $G$ is an increasing function.

  Now, given any $\rho_1\in(0,\rho_0)$, we can take
  \[t_1=-e^mG(\rho_1)=e^m\int_{\rho_1}^{\rho_0}\frac{d\rho}{\rho f(\rho)}<+\infty.\]
  Then \eqref{eq:rhoint} implies
  $G(\rho(t))\leq G(\rho_1)$, which leads to $\rho(t)\leq \rho_1$.
\end{proof}

Next, we focus on the behaviors of the dynamics of $d$, which varies
in different regions of initial data.

\subsection{Blow-up of type I supercritical initial data}
For any type I supercritical initial data \eqref{eq:super1}, we have
the following Lemma on a positive lower bound of $d$.
\begin{lemma}\label{lem:dlow}
  Consider \eqref{eq:drho} with  type I supercritical initial data
  satisfying \eqref{eq:super1}.
  Then there exists $c>0$ such that $d(t)>c$ in the whole lifespan of $d$.
\end{lemma}
\begin{proof}
  We express the trajectory of the dynamics in the $(\rho,d)$ phase
  plane by $d=\d(\rho)$, and compare the function $\d$ with
  the threshold function $\sigma$.
  Since $\d(\rho_0)=d_0>\sigma(\rho_0)$, we have
  \begin{equation}\label{eq:dgtrsigma}
    \d(\rho)>\sigma(\rho)>0,
  \end{equation}
  for any $\rho\in(0,\rho_0]$ that lies in the domain of $\d$.
  Here, the second inequality is due to Proposition
  \ref{prop:sigmapos}.
  Therefore, we have
  \begin{equation}\label{eq:dinf}
    \inf_td(t)=\inf_\rho \d(\rho)\geq0,
  \end{equation}
  where the equality can only be attained in the case when
  $\lim_{\rho\to0}\d(\rho)=0$, namely that $(0,0)$ lies on the
  trajectory. Moreover, from \eqref{eq:dgtrsigma} we have
  \begin{equation}\label{eq:dzerogtr}
    \d'(0)\ge\sigma'(0)=\beta.
  \end{equation}
  However, in view of Proposition \ref{prop:initiald}, any trajectory
  with $\d(0)=0$ must have either $\d'(0)=0$ or $\d'(0)=\beta$. The former
  contradicts with \eqref{eq:dzerogtr}. For the latter case, it has
  been shown in Theorem \ref{thm:sigma} that $\sigma$ is the only
  trajectory that enters $(0,0)$ with the slope $\beta$. Hence, the
  equality in \eqref{eq:dinf} cannot be reached, finishing the proof.
\end{proof}

With the uniform lower bound, we are ready to show that $d$ must blow
up in finite time.
Let us rewrite the dynamics of $d$ in \eqref{eq:drho} as
\begin{equation}\label{eq:ddym}
\dot d = -e^{-\rhob}\Big(f''(\rho)d^2+\big(f(\rho)+2\rho f'(\rho)\big)d+\rho^2 f(\rho)\Big) = C(\rho)\big(d-\d_-(\rho)\big)\big(d-\d_+(\rho)\big),
\end{equation}
where 
\[
  C(\rho)=-e^{-\rhob}f''(\rho),\,
  \d_\pm(\rho) = \frac{-\big(f(\rho) + 2 \rho f'(\rho)\big)
  \mp \sqrt{\big(f(\rho) + 2 \rho f'(\rho)\big)^2-4\rho^2f(\rho)f''(\rho)}}{2f''(\rho)}.
\]
Observe that when $\rho\in(0,\rho_c)$, the coefficient $C(\rho)>0$. 
The curves $(\rho, \d_\pm(\rho))$ are the two nullclines of $d$ in the
phase plane. We have $\d_-(\rho)<0<\d_+(\rho)$, and furthermore
\begin{equation}\label{eq:nclimit}
  \lim_{\rho\to0+}\d_\pm(\rho)=0.
\end{equation}
Therefore, when $\rho$ is small enough, the dynamics of $d$ would
behave like $\dot d\sim C(0)d^2$, which leads to a finite time blowup.

\begin{proof}[Proof of Theorem \ref{thm:drho}(a)]
  First, from \eqref{eq:nclimit} we can find a small $\rho_1>0$ such that
  \[2\d_+(\rho_1)<c,\]
  where $c$ is the uniform lower bound of $d$ as in Lemma \ref{lem:dlow}.

  From Lemma \ref{lem:rhodown}, there exists a finite time $t_1$
  such that $\rho(t_1)=\rho_1$, unless $d$ already blows up before $t_1$.

  We focus on the dynamics \eqref{eq:ddym} starting from $t_1$.
  From the hypotheses of $f$ in \eqref{eq:f}, we can obtain a uniform
  lower bound on $C(\rho)$ as
  \[C(\rho)\geq e^{-m}\cdot\min_{\rho\in[0,\rho_1]}\big(-f''(\rho)\big)=:C>0,\quad\forall~\rho\in[0,\rho_1].\]
  Since $\rho(t)\in(0,\rho_1]$ for any $t\geq t_1$, we deduce from
  \eqref{eq:ddym} that
  \[\dot d\geq C \big(d-\d_-(\rho)\big) \big(d-\d_+(\rho)\big)
  \geq C \big(d-\d_+(\rho)\big)^2\geq C \big(d-\d_+(\rho_1)\big)^2,\]
 with initial condition at $t_1$ satisfying
 \[d(t_1)>c>2\d_+(\rho_1),\]
 where we have used Lemma \ref{lem:dlow}.
 Solving the initial value problem would yield
 \[d(t)>\d_+(\rho_1)+\frac{\d_+(\rho_1)}{1-C\d_+(\rho_1)(t-t_1)}.\]
 Therefore, we must have
 \[\lim_{t\to t_*-}d(t)=+\infty\]
 at a finite time
 \[t_*\leq t_1+\frac{1}{C\d_+(\rho_1)}.\]
\end{proof}

\subsection{Blow-up of type II supercritical initial data}
For type II supercritical initial data \eqref{eq:super2},
the blowup is a direct consequence of a comparison with the threshold
function $\gamma$.

\begin{proof}[Proof of Theorem \ref{thm:drho}(b)]
  Compare the trajectory $d=\d(\rho)$ with the threshold function
  $\gamma$. Since $\d(\rho_0)=d_0\le\gamma(\rho_0)$, we have
  \[\d(\rho)\le\gamma(\rho),\]
  for any $\rho\in(\rhoc,\rho_0]$ that lies in the domain of $\d$.
  Since $\lim_{\rho\to\rhoc+}\gamma(\rho)=-\infty$, $\d$ must blow up
  to $-\infty$ at some $\rho_1\ge\rhoc$.
  Apply Lemma \ref{lem:rhodown}, $d(t)$ must blow up to at a finite
  time $t_1$.
\end{proof}

\subsection{Global regularity of subcritical initial data}
Let us first consider initial data that lie in the region
\eqref{eq:sub1}. The main idea is that the coefficient $C(\rho)$ in
the dynamics \eqref{eq:ddym} has a favorable sign
that prevents $d$ from going to $-\infty$.
On the other hand, $d$ is controlled from above by $\sigma$.

\begin{proof}[Proof of Theorem \ref{thm:drho}(c) for region \eqref{eq:sub1}]
  We start with obtaining a lower bound on $d$.
  First, suppose $\rho_0<\rhoc$. In view of the definition of $\d_-$,
  it has a uniform lower bound on $[0,\rho_0]$
  \[\min_{\rho\in[0,\rho_0]}\d_-(\rho)\geq \underline{d}>-\infty,\]
  where $\underline{d}$ depends on $f$ and $\rho_0$.
  Since $C(\rho)>0$ for $\rho\in[0,\rho_0]$, \eqref{eq:ddym} implies
  $\dot d>0$ as long as $d<\underline{d}$. This implies the lower
  bound
  \[d\geq\min\{d_0,\underline{d}\}\]
  in the whole timespan of $d$.
  
  For $\rho_0=\rhoc$, it is easy to verify that the dynamics
  \eqref{eq:ddym} is locally wellposed. Hence, there exists a time
  $t_1>0$ such that solution $d(t_1)$ exists. From the monotonicity of
  $\rho$ in time, we know $\rho(t_1)<\rhoc$. Hence the dynamics
  starting from $t_1$ reduces to the prior case, leading to a lower
  bound of $d$.

To obtain an upper bound of $d$, note that $\d(\rho_0)=d_0\leq
\sigma(\rho_0)$.
We compare the trajectory $\d$ with $\sigma$ and get
\begin{equation}\label{eq:dlessigma}
  \d(\rho)\leq\sigma(\rho).
\end{equation}
From Proposition \ref{prop:sigmablow}, we know $\sigma$ is bounded in
$[0,\rhoc]$.
Since $\rho(t)\leq \rho_0$ for all time, we conclude that
\[d(t)\leq \max_{\rho\in[0,\rho_0]}\d(\rho)\leq \max_{\rho\in[0,\rhoc]}\sigma(\rho)<+\infty.\]
\end{proof}

For initial data that lie in the region \eqref{eq:sub2},
note that $\sigma$ is not necessarily defined for large
$\rho$. Instead, we may bound $d$ from above using the dynamics
\eqref{eq:ddym}.
When $\rho\in(\rhoc,1)$, the coefficient $C(\rho)<0$. It is possible
that the quadratic form in \eqref{eq:ddym} has no real roots, and
$\d_\pm$ do not exist. In which case we have $\dot d<0$.
If $\d_\pm$ exist, we still have $\dot d<0$ when $d>\d_-\geq\d_+$.
Hence, $d$ can not blow up to $+\infty$.
The lower bound can be controlled by $\gamma$.

\begin{proof}[Proof of Theorem \ref{thm:drho}(c) for region \eqref{eq:sub2}]
  For the upper bound, we apply Proposition \ref{prop:sigmablow}.
  If $\sigma$ is well-defined in $[0,\rho_0]$, then
  $d$ is bounded by $\sigma$ by comparison \eqref{eq:dlessigma}.
  If $\sigma$ blows up at $\rho_*\in(\rhoc,\rho_0)$, we observe that
  $\d_-$, if exists, has a uniform upper bound on $[\rho_*,\rho_0]$
  \[\max_{\rho\in[\rho_*,\rho_0]\cap {\text{Dom}(\d)}}\d_-(\rho)\leq \overline{d}<+\infty,\]
  where $\overline{d}$ depends on $f$, $\rho_*$ and $\rho_0$. 
  Since $C(\rho)<0$ for $\rho\in[\rho_*, \rho_0]$,
  \eqref{eq:ddym} implies $\dot d<0$ as long as $d>\overline{d}$. This
  leads to the upper bound
  \[d\leq\max\{d_0, \overline{d}\}.\]
  Once $\rho(t)$ drops below $\rho_*$, it can be controlled by
  $\sigma$.

  For the lower bound, can compare the trajectory $\d$ with the
  threshold function $\gamma$. Since $\d(\rho_0)=d_0>\gamma(\rho_0)$,
  we have
  \[\d(\rho)>\gamma(\rho),\quad \forall~\rho\in[\rhoc,\rho_0].\]
  Hence, $d$ has a lower bound as long as $\rho>\rhoc$.
  Moreover, $\d(\rhoc)$ is bounded. This is because if
  $\lim_{\rho\to\rhoc+}\d(\rho)=-\infty$, we deduce from Theorem
  \ref{thm:gamma} that $\d=\gamma$, violating the initial condition
  \eqref{eq:sub2}.
  Therefore, by Lemma \ref{lem:rhodown}, there exists a time $t_1$
  such that $\rho(t_1)=\rhoc$ and $d(t_1)=\d(\rhoc)$ is bounded. The
  dynamics enter the region \eqref{eq:sub1} at $t_1$. The global
  behavior follows from the proof for the region \eqref{eq:sub1}.
\end{proof}

\bibliographystyle{plain}
\bibliography{traffic}

\begin{thebibliography}{10}

\bibitem{bhatnagar2021critical}
Manas Bhatnagar, Hailiang Liu, and Changhui Tan.
\newblock Critical thresholds in the {E}uler-{P}oisson-alignment system.
\newblock {\em arXiv preprint arXiv:2111.11999}, 2021.

\bibitem{bressan2020traffic}
Alberto Bressan and Wen Shen.
\newblock On traffic flow with nonlocal flux: a relaxation representation.
\newblock {\em Archive for Rational Mechanics and Analysis}, 237(3):1213--1236,
  2020.

\bibitem{carrillo2016critical}
Jos{\'e}~A Carrillo, Young-Pil Choi, Eitan Tadmor, and Changhui Tan.
\newblock Critical thresholds in {1D} {E}uler equations with nonlocal forces.
\newblock {\em Mathematical Models and Methods in Applied Sciences},
  26(1):185--206, 2016.

\bibitem{colombo2018nonlocal}
Rinaldo~M Colombo and Elena Rossi.
\newblock Nonlocal conservation laws in bounded domains.
\newblock {\em SIAM Journal on Mathematical Analysis}, 50(4):4041--4065, 2018.

\bibitem{dafermos2016hyperbolic}
Constantine~M Dafermos.
\newblock {\em {Hyperbolic conservation laws in continuum physics; 4th ed.}}
\newblock Grundlehren der mathematischen Wissenschaften : a series of
  comprehensive studies in mathematics. Springer, Berlin, 2016.

\bibitem{drake1965statistical}
Jennifer Drake, Joseph Schofer, and ADJR May.
\newblock A statistical analysis of speed-density hypotheses.
\newblock {\em Traffic Flow and Transportation}, 1965.

\bibitem{engelberg2001critical}
Shlomo Engelberg, Hailiang Liu, and Eitan Tadmor.
\newblock Critical thresholds in {E}uler-{P}oisson equations.
\newblock {\em Indiana University Mathematics Journal}, 50:109--157, 2001.

\bibitem{keimer2017existence}
Alexander Keimer and Lukas Pflug.
\newblock Existence, uniqueness and regularity results on nonlocal balance
  laws.
\newblock {\em Journal of Differential Equations}, 263:4023--4069, 2017.

\bibitem{keimer2018nonlocal}
Alexander Keimer, Lukas Pflug, and Michele Spinola.
\newblock Nonlocal scalar conservation laws on bounded domains and applicatons
  in traffic flow.
\newblock {\em SIAM Journal of Mathematical Analysis}, 50(6):6271--6306, 2018.

\bibitem{kurganov2009non}
Alexander Kurganov and Anthony Polizzi.
\newblock Non-oscillatory central schemes for traffic flow models with
  {A}rrhenius look-ahead dynamics.
\newblock {\em NHM}, 4(3):431--451, 2009.

\bibitem{lee2019thresholds}
Yongki Lee.
\newblock Thresholds for shock formation in traffic flow models with
  nonlocal-concave-convex flux.
\newblock {\em Journal of Differential Equations}, 266(1):580--599, 2019.

\bibitem{lee2013thresholds}
Yongki Lee and Hailiang Liu.
\newblock Thresholds in three-dimensional restricted euler--poisson equations.
\newblock {\em Physica D: Nonlinear Phenomena}, 262:59--70, 2013.

\bibitem{lee2015thresholds}
Yongki Lee and Hailiang Liu.
\newblock Thresholds for shock formation in traffic flow models with
  {A}rrhenius look-ahead dynamics.
\newblock {\em Discrete \& Continuous Dynamical Systems-A}, 35(1):323--339,
  2015.

\bibitem{lee2019sharp}
Yongki Lee and Changhui Tan.
\newblock A sharp critical threshold for a traffic flow model with look-ahead
  dynamics.
\newblock {\em arXiv preprint arXiv:1905.05090}, 2019.

\bibitem{lighthill1955kinematic}
M.~J. Lighthill and G.~B. Whitham.
\newblock On kinematic waves. ii. a theory of traffic flow on long crowded
  roads.
\newblock {\em Proc. Roy. Soc. London Ser. A}, 229:317--345, 1955.

\bibitem{pipes1966car}
Louis~Albert Pipes.
\newblock Car following models and the fundamental diagram of road traffic.
\newblock {\em Transportation Research/UK/}, 1966.

\bibitem{richard1956highway}
P.~I. Richards.
\newblock Shock waves on the highway.
\newblock {\em Operations Research}, 4(1):42--51, 1956.

\bibitem{sopasakis2006stochastic}
Alexandros Sopasakis and Markos~A Katsoulakis.
\newblock Stochastic modeling and simulation of traffic flow: asymmetric single
  exclusion process with {A}rrhenius look-ahead dynamics.
\newblock {\em SIAM Journal on Applied Mathematics}, 66(3):921--944, 2006.

\bibitem{sun2020class}
Yi~Sun and Changhui Tan.
\newblock On a class of new nonlocal traffic flow models with look-ahead rules.
\newblock {\em Physica D: Nonlinear Phenomena}, 413:132663, 2020.

\bibitem{tadmor2003critical}
Eitan Tadmor and Hailiang Liu.
\newblock Critical thresholds in {2D} restricted {E}uler-{P}oisson equations.
\newblock {\em SIAM Journal on Applied Mathematics}, 63(6):1889--1910, 2003.

\bibitem{tadmor2014critical}
Eitan Tadmor and Changhui Tan.
\newblock Critical thresholds in flocking hydrodynamics with non-local
  alignment.
\newblock {\em Philosophical Transactions of the Royal Society A: Mathematical,
  Physical and Engineering Sciences}, 372(2028):20130401, 2014.

\bibitem{tadmor2021critical}
Eitan Tadmor and Changhui Tan.
\newblock Critical threshold for global regularity of
  {E}uler-{M}onge-{A}mp\`ere system with radial symmetry.
\newblock {\em arXiv preprint arXiv:2108.00120}, 2021.

\bibitem{tan2020euler}
Changhui Tan.
\newblock On the {E}uler-alignment system with weakly singular communication
  weights.
\newblock {\em Nonlinearity}, 33(4):1907, 2020.

\end{thebibliography}

\end{document}